\documentclass[twoside,a4paper,reqno,11pt]{amsart} 
\usepackage[top=28mm,right=30mm,bottom=28mm,left=30mm]{geometry}

\usepackage{amsfonts, amsmath, amssymb, mathrsfs, bm, latexsym, stmaryrd, array, mathtools, bold-extra,xcolor}

\usepackage{etaremune}
\usepackage{enumerate}
\usepackage[]{hyperref}

\newcommand{\la}{\langle}
\newcommand{\ra}{\rangle}

\newcommand{\leqs}{\leqslant}
\newcommand{\geqs}{\geqslant}

\newcommand{\Aut}{\operatorname{Aut}}

\newcommand{\CC}{\mathbb{C}}

\newcommand{\Irr}{\mathrm {Irr}}

\newcommand{\OO}{\mathrm {(O)}}

\newcommand{\z}{\mathbf{Z}}
\newcommand{\F}{\mathrm{(F)}} 
\newcommand{\CI}{\mathrm{(CI)}} 
\newcommand{\Fp}{\mathrm{(F')}}

\makeatletter
\newcommand{\imod}[1]{\allowbreak\mkern4mu({\operator@font mod}\,\,#1)}
\makeatother

\renewcommand{\leq}{\leqs}
\renewcommand{\geq}{\geqs}

\newtheorem{theorem}{Theorem}

\newtheorem{corol}[theorem]{Corollary}

\newtheorem{thm}{Theorem}[section] 
\newtheorem{prop}[thm]{Proposition} 
\newtheorem{lem}[thm]{Lemma}
\newtheorem{cor}[thm]{Corollary}

\theoremstyle{definition}
\newtheorem{rem}[thm]{Remark}

\newtheorem{problem}{Problem}

\begin{document}
\title[Camina pairs and cosets]{Some generalizations of Camina Pairs and orders of elements in cosets}

\author{Thu T.H. Quan}
\address{Thu T.H. Quan, Department of Mathematics and Statistics, Binghamton University, Binghamton, NY 13902-6000, USA}
\email{tquan1@binghamton.edu}

\author{Hung P. Tong-Viet}
\address{H.P. Tong-Viet, Department of Mathematics and Statistics, Binghamton University, Binghamton, NY 13902-6000, USA}
\email{htongvie@binghamton.edu}

\renewcommand{\shortauthors}{Quan, Tong-Viet}

\begin{abstract}
In this paper, we investigate certain generalizations of Camina pairs. Let $H$ be a nontrivial proper subgroup of a finite group $G$. We first show that every nontrivial irreducible complex character of $H$ induces homogeneously to $G$ if and only if for every $x\in G\setminus H$, the element $x$ is conjugate to $xh$ for all $h\in H$. Furthermore we prove that if $xh$ is conjugate to either $x$ or $x^{-1}$ for all $h\in H$ and all $x\in G\setminus H$, then the normal closure $N$ of $H$ in $G$ also satisfies the same condition, and $N$ is nilpotent. Finally, we determine the structure of $H$ under the assumption that for every element $x\in G\setminus H$ of odd order, the coset $xH$ consists entirely of elements of odd order.
\end{abstract}

\date{\today}

\maketitle

\section{Introduction}\label{s:intro}
Let $G$ be a finite group and let $H$  be a nontrivial proper normal subgroup of $G$. The pair $(G,H)$ is called a \emph{Camina pair} if, for every element $g\in G\setminus H$, the element $g$ is conjugate to every element of the coset $gH$. The study of Camina pairs has a long history, beginning with Camina's work in \cite{Camina}, where he introduced certain conditions that nearly characterize Frobenius groups. In particular, Camina showed in  \cite[Theorem 2]{Camina} that if $(G,H)$ is a Camina pair, then one of the following holds:  $G$ is a Frobenius group with Frobenius kernel $H$; or $H$ is a $p$-group; or $G/H$ is a $p$-group for some prime $p$. We refer the reader to the recent survey  by Lewis \cite{Lewis} for a comprehensive discussion of  various equivalent formulations of Camina pairs and their generalizations.

For a finite group $G$, we denote by $\Irr(G)$ the set of complex irreducible characters of $G$. Using character theory of finite groups, an equivalent formulation of the Camina pair was established by Chillag and Macdonald in \cite[Propsition 3.1]{Chillag-Macdonald}. They show that if $H$ is not contained in the kernel of an irreducible character $\chi\in\Irr(G)$, then $\chi$ vanishes on $G\setminus H$. Another equivalent condition was also provided by Kuisch and van der Waall in \cite{Kuisch-vanderWaall}. Recall that a character $\theta\in \Irr(H)$ is said to induce \emph{homogeneously} to $G$ if $\theta^G = a\xi$ for some $\xi\in \Irr(G)$ and a positive integer $a$, where $\theta^G$ denotes the character of $G$ induced from $\theta$. Kuisch and van der Waall proved in \cite[Lemma 2.1]{Kuisch-vanderWaall} that $(G,H)$ is a Camina pair if and only if every nontrivial irreducible character of $H$ induces homogeneously to $G$. 
In fact, they considered a more general assumption where $H$ is not required to be a normal subgroup of $G$. They referred to this broader condition as  $\CI$, defined as follows:

\smallskip
\noindent
{{$\CI$:} $1<H<G$, and every nontrivial irreducible character of $H$ induces homogeneously to $G$.}

\smallskip
In the case when $H$ is not normal in $G$ and the pair $(G,H)$ satisfies condition $\CI$, among other results, it is proved in \cite{Kuisch-vanderWaall} that $H$ is nilpotent and if $N$ is the normal closure of $H$ in $G$, that is, $N$ is the smallest normal subgroup of $G$ containing $H$, then $N$ is a $p$-group for some prime $p$ or $G$ is a Frobenius group with Frobenius kernel $N.$

 \smallskip
In this paper, we also consider the following condition, which we refer to as the $\F$ condition. This condition is precisely the Camina pair condition without the assumption that $H$ is normal in $G$, and is defined as follows:

\smallskip
\noindent
$\F:$ $1<H<G$ and for all $x\in G\setminus H$, $x$ is conjugate to $xh$ for all $h\in H$. 

For $g\in G$, let $g^G$ denote the conjugacy class of $G$ containing $g$. Then condition $\F$ can be stated as follows: 
\begin{center} $1<H<G$ and for all $x\in G\setminus H,$ $xH\subseteq x^G$. \end{center}
  
\smallskip
Our first theorem establishes the equivalence of the $\CI$ and $\F$ conditions for any finite group $G$ and nontrivial proper subgroup $H$.

\begin{theorem}\label{thm1}
	Let $G$ be a finite group and let $H$ be a nontrivial proper subgroup of $G$. Then the pair $(G,H)$ satisfies condition $\CI$ if and only if it satisfies condition $\F$. 
\end{theorem}

In a recent paper, Akhlaghi and Beltr\'{a}n \cite{Akhlaghi-Beltran} studied the following generalization of Camina pairs: let  $H$ be a nontrivial proper normal subgroup  of a finite group $G$, and suppose that for all $g\in G\setminus H$, the element $gn$ is conjugate to $g$ or $g^{-1}$ for every $n\in  H.$ Under this assumption, they proved that if $H$ is not nilpotent,  then $G/H$ is a $p$-group and $G$ is $p$-nilpotent and solvable for some prime $p.$ On the other hand, if $H$ is nilpotent, then $G$ is a Frobenius group with kernel $H$; or $H$ is of prime power order; or $|G/H|=2$, the center of $G$ is of order $2$, is contained in $H$ and $G/Z(G)$ is a generalized dihedral group.

We extend the condition studied by Akhlaghi and Beltr\'{a}n by removing the normality requirement on $H$. Specifically, we consider the following condition, which we denote by $\Fp$. 

\smallskip
{$\Fp$}: $1<H<G$, and for all $x\in G\setminus H$,  $xh$ is conjugate to $x$ or $x^{-1}$ for every $h\in H$. 
\smallskip

Note that condition $\F$ implies condition $\Fp$. As it turns out, if the pair $(G,H)$ satisfies condition $\Fp$, then the pair $(G,N)$ also satisfies $\Fp$, where $N$ is the normal closure of $H$ in $G$. Moreover, under this assumption, $N$ is nilpotent. 

\begin{theorem}\label{thm2}
	Let $G$ be a finite group and let $H$ be a nontrivial proper subgroup of $G$. Suppose that $H$ is not normal in $G$.  Let $N$ be the normal closure of $H$ in $G$. If the pair $(G,H)$ satisfies condition $\Fp$, then $(G,N)$ satisfies $\Fp$ and $N$ is nilpotent. 
\end{theorem}

We should point out the connection between conditions $\F$ and $\Fp$ and recent work by Camina et al. in \cite{Camina2}. Let $G$ be a finite group and let $N$ be a nontrivial proper normal subgroup of $G$. The pair $(G,N) $ is called an \emph{equal order pair} if, for every $x\in G\setminus N$, the elements in the coset $xN$ all have the same order as $x$. Clearly, if the pair $(G,H)$ satisfies condition $\F$ or $\Fp$, then the pair $(G,N)$ is an equal order pair, where $N$ is the normal closure of $H$ in $G$.  In particular, since $N$ is nilpotent by Theorem \ref{thm2}, the structure of $G$ is constrained by \cite[Theorem B]{Camina2}. For instance, if $N$ is a $p$-group for some prime $p$, then for every prime $r\neq p$, the Sylow $r$-subgroups of $G$ are either cyclic or generalized quaternion; and if $N$ is not of prime power order, then $N$ coincides with the Fitting subgroup of $G$, that is, the largest nilpotent normal subgroup of $G$, and the Sylow subgroups of $G/N$ are either cyclic or generalized quaternion. This naturally lead to the following problem.

\begin{problem}\label{prob1}
Let $G$ be a finite group and let $H$ be a nontrivial proper subgroup of $G$. Study the structure of $G$ if the pair $(G,H)$ satisfies the condition: for every $x\in G\setminus H$, all elements in the coset $xH$ have the same order as $x$. 
\end{problem}

\medskip

The  main inspiration  for our definition  of condition $\F$ follows from the work of Navarro and Guralnick in \cite{GN}.  In that paper, the authors investigate the structure of finite group $G$ with a conjugacy class $K$ such that $K^2$ is a conjugacy class of $G$. This is a special case of a conjecture proposed by Arad and Herzog in  \cite{Arad-Herzog} stating that in a finite nonabelian simple group, the product of two nontrivial conjugacy classes cannot be a single class. In \cite{GN}, they proved that if $K^2$ is a conjugacy class of $G$, then $N:=[x,G]$ is solvable, where $x\in K$. Moreover, all elements of the coset $xN$ are conjugate in $G$. Inspired by this, they also show that  if $N$ is a normal subgroup of $G$ and $x\in G$ and that either all elements of $xN$ are $G$-conjugate or all elements of $xN$ have odd order, then $N$ is solvable.  As before, dropping the normality assumption, we can prove the following result.

\begin{theorem}\label{th:odd-order-coset}
Let $G$ be a finite group and let $H$ be a proper subgroup of $G$. Suppose that for every element $x\in G\setminus H$ of odd order, all elements in the coset $xH$ have odd order. Then either $O^2(H)$ is normal in $G$ and $G/O^2(H)$ is a $2$-group, or $H$ is solvable.
\end{theorem}
As a consequence, we obtain the following.

\begin{corol}\label{cor1} Let $G$ be a finite group and let $H$ be a proper subgroup of $G$.  Suppose that for every $x\in G\setminus H$, all elements in the coset $xH$ have the same order.  Then  either $H$ is solvable or the following hold.
\begin{enumerate}[$(a)$]
\item $O^2(G)\leq H$ and $H$ is subnormal in $G$.
\item Every element in $G\setminus H$ is a $2$-element.
\item The pair $(N_G(H),H)$ is an equal order pair.
 \end{enumerate}
\end{corol}

Using a similar argument to that employed in the proof of Theorem~\ref{th:odd-order-coset}, we obtain the following variant of the Baer-Suzuki theorem. Recall that if $p$ is a prime, then an element  $g\in G$ is called $p$-regular if its order is not divisible by $p.$

\begin{theorem}
    \label{cor2}
Let $G$ be a finite group and let $p$ be a prime. Suppose that  $x\in G$ is a $p$-element such that $xy$ is $p$-regular for every nontrivial $p$-regular element $y\in G$. Then $x\in O_p(G)$.
\end{theorem}

This result may be known to experts. In fact, it has been shown that if $x\in G$ and $xy$ is $p$-regular for every $p$-regular element $y\in G$, then $x\in O_{p'}(G)$. This was independently proved by Robinson \cite{Robinson} and Flavell \cite{Flavell}. Robinson's proof is character-theoretic, while Flavell's argument is  group-theoretic. Both proofs do not use the classification of finite simple groups. Our proof of Theorem \ref{cor2} is also classification-free. Theorem \ref{cor2} can be considered a dual of Theorem D in \cite{BLM}, where it is shown that if $x\in G$ is  $p$-regular such that the order of $xy$ is divisible by $p$ for all nontrivial $p$-elements $y\in G$, then $x\in O_{p'}(G)$.  This is proved modulo a conjecture on finite almost simple groups.

\medskip
\noindent
\textbf{Notation.} Our notation is standard. We follow \cite{Isaacsbook} for character theory of finite groups. If $G$ is a finite group, then $\Irr(G)$ denotes the set of its complex irreducible characters. If $N$ is a subgroup of $G$ and $\theta\in \Irr(N)$, then $\Irr(G\mid\theta)$  denotes the set of  irreducible constituents of the induced character $\theta^G$. By the Frobenius reciprocity, this is the set of  irreducible characters $\chi\in\Irr(G)$ such that the restriction $\chi_N$ contains $\theta$ as a  constituent. In this case, we say that $\chi$ `lies over' $\theta$ and $\theta$ `lies under' $\chi$. If $g\in G$, then we can uniquely write $g=g_pg_{p'}$, where $g_p,g_{p'}\in \la g\ra$ have orders a power of $p$ and coprime to  $p$, respectively. For a proper subgroup $H$ of $G$, we define $\Delta_H(G)=G \setminus \cup_{g\in G}H^g$, where $H^g=g^{-1}Hg$ is a conjugate of $H$ by an element $g\in G$.  Finally, we write $\pi(G)$ for the set of distinct prime divisors of the order of $G$.

\section{Preliminaries}

In this section, we collect and prove some auxiliary results that will be used in the proofs of the main theorems. We begin by describing some  properties of groups that satisfy condition $\CI$.

\begin{lem}\emph{(\cite[Lemma 1.3]{Kuisch-vanderWaall}).}\label{lemma1.3}
Let $G$ be a finite group and let $H$ be a subgroup of $G$. Suppose that the pair $(G,H)$ satisfies condition $\CI$. Let $M$ be a normal subgroup of $G$ such that  $H\not\subseteq M$. Then $M <H$ and $(G/M, H/M)$ satisfies condition $\CI$. 
\end{lem}

\begin{lem} \emph{(\cite[Theorem 3.2]{Kuisch-vanderWaall}).}\label{lem:Cam}
Let $G$ be a finite group and let $H$ be a subgroup of $G$.	Suppose that  the pair $(G,H)$ satisfies condition $\CI$. Let $N$ be the normal closure of $H$ in $G$. Then $(G,N)$ is a Camina pair. 
\end{lem}

\begin{lem}\label{thm4.1}
Let $G$ be a finite group and let $H$ be a subgroup of $G$.	Suppose that  the pair $(G,H)$ satisfies condition $\CI$ and that $H$ is not normal in $G$. Let $N$ be the normal closure of $H$ in $G$. Then $N$ is nilpotent.
\end{lem}
\begin{proof}
	By \cite[Theorem 4.1]{Kuisch-vanderWaall}, either $N$ is a $p$-group for some prime $p$ or $G$ is a Frobenius group with Frobenius kernel $N.$ Clearly $N$ is nilpotent if it is a $p$-group. If  $N$ is a Frobenius kernel, then it is nilpotent by Thompson's theorem \cite[Theorem 1]{Thompson}.  
\end{proof}

\begin{lem} \label{Hnormal}
Let $G$ be a finite group and let $H$ be a subgroup of $G$.	Suppose that the pair $(G,H)$ satisfies condition $\CI$ and $H$ is not normal in $G$. Let $N$ be the normal closure of $H$ in $G$. Then  $H$ is a normal subgroup of $N$ and $N/H$ is abelian. 
\end{lem}
\begin{proof}
	From Lemma \ref{thm4.1},  $N$ is nilpotent, which implies that $N'$ is a proper characteristic subgroup of $N$. Note that $N$ is nontrivial as it contains $H>1$. Since $N$ is normal in $G$, $N'$ is also normal in $G$. In addition, as $N$ is the smallest normal subgroup of $G$ containing $H$, $H\nsubseteq N'$.  Lemma \ref{lemma1.3} yields that  $N' < H$. Thus $H$ is a normal subgroup of $N$ and  $N/H$ is abelian. 
\end{proof}

Now suppose the pair $(G,H)$ satisfies condition $\F$, where $H$ is a subgroup of a finite group $G$.  We proceed to establish a sequence of results that will lead to the proof of the sufficient direction of Theorem \ref{thm1}. These results are analogous to those obtained by Kuisch and van der Waall, as discussed above.

\begin{lem} \label{lemma6}
Let $G$ be a finite group and let $H$ be a subgroup of $G$.	Suppose that the pair $(G,H)$ satisfies condition $\F$. Let $M$ be a normal subgroup of $G$ such that $H\nsubseteq M$. Then $M < H$. Moreover, $(G/M, H/M)$ satisfies condition $\F$. 
\end{lem}

\begin{proof}
Suppose that $M\nsubseteq H$, and  let  $x\in M\setminus H$. Since $(G,H)$ satisfies condition $\F$ and $M$ is normal in $G$, we have $xH\subseteq x^G\subseteq M$. It follows that $H\subseteq x^{-1}M = M$, a contradiction. Therefore, $M\subseteq H$. Since $H\not\subseteq M$, we conclude that $M<H$, as required. The final claim follows immediately, as conjugacy is preserved under taking quotients. 
\end{proof}

\begin{lem} \label{corol7}
Let $G$ be a finite group and let $H$ be a subgroup of $G$.	Suppose that the pair $(G,H)$ satisfies condition $\F$. Then $\z(G)\leq H\leq G'$.  
\end{lem}
\begin{proof}
If $x$ is an element in $ \z(G)$ that is not in $H$, then by condition $\F$, $xH \subseteq x^G = \left\{x\right\}$, which implies $H = 1$, a contradiction. Hence, every element of $\z(G)$ is contained in $H$, and we establish the containment $\z(G)\leq H$. Next,  take an element $y\in  G\setminus H$. By the assumption that $(G,H)$ satisfies condition $\F$, we have  $yH \subseteq y^G$. That is, for each $h\in H$, there exists $g\in G$ such that $yh = y^g$, and hence $$h = y^{-1}y^g = [y,g] \in G'.$$ It follows that $H\leq G'$, as desired.   
\end{proof}

\begin{lem}\label{lemma9}
Let $G$ be a finite group and let $H$ be a subgroup of $G$. Suppose that the pair $(G,H)$ satisfies condition $\F$ and $G$ is a nilpotent group. Then $H$ is a normal subgroup of $G$.
\end{lem}
\begin{proof}
	Since $G$ is nilpotent, it possesses an upper central series
	\begin{align*}
		1 = Z_0 < Z_1 = \z(G) <Z_2 <\ldots <  Z_{k-1}<Z_k = G.
	\end{align*}
	Because $Z_1 = \z(G)$ is a nontrivial normal subgroup of $G$, by Lemma \ref{corol7},  $Z_1\leq H$. If $Z_1 = H$, then the claim follows. Suppose instead that $Z_1<H$.   Then, by Lemma \ref{lemma6}, the pair $\left(G/Z_1,H/Z_1\right)$ satisfies condition $\F$. Applying Lemma \ref{corol7} to the pair $(G/Z_1, H/Z_1)$ yields $Z_2/Z_1=\z(G/Z_1)\leq H/Z_1$ and hence  $Z_2\leq H$. If $Z_2 = H$, then we are done. Otherwise, we continue the process:  the pair $(G/Z_2, H/Z_2)$ satisfies condition $\F$ so $Z_3/Z_2 = \z(G/Z_2) \leq H/Z_2$, and thus $Z_3\leq H$.  By repeating this argument, we either eventually reach  $H=Z_i$ for some $i$, in which case $H\unlhd G$ or else, we find that  $Z_{k-1}\leq H$.  Since $G'\leq Z_{k-1}\leq H$, $H$ is normal in $G$ as well.
\end{proof}

\begin{lem}\label{lemma11}
Let $G$ be a finite group and let $H$ be a subgroup of $G$.	Suppose that the pair $(G,H)$ satisfies condition $\F$ and that $H$ is not normal in $G$. Let $N$ be the normal closure of $H$  in $G$. Then $N=\cup_{g\in G}H^g$. Consequently,  $1<N<G$. 
\end{lem}
\begin{proof}
	By definition, the normal closure $N$ of $H$ in $G$ is generated by all conjugates of $H$, i.e., by the set $\cup_{g\in G}H^g$. To prove the lemma, it suffices to show that this union is closed under multiplication. Let $g_1,g_2\in G$ and $h_1,h_2\in H$. Consider the product $$h_1^{g_1}h_2^{g_2}=g_1^{-1}h_1g_1g_2^{-1}h_2g_2.$$ Set $x=g_1g_2^{-1}$. We rewrite this as:
	\[	 h_1^{g_1}h_2^{g_2}=g_2^{-1}(g_2g_1^{-1}h_1g_1g_2^{-1}h_2)g_2 = \left(h_1^{x}h_2\right)^{g_2}. \]
	If $h_1^{x} \in H$, then 
	\[	h_1^{g_1}h_2^{g_2}=\left(h_1^{x}h_2\right)^{g_2}  \in (Hh_2)^{g_2} = H^{g_2} \subseteq\cup_{g\in G}H^g.\]
	If $h_1^{x} \notin H$, then by condition (F), 
	\[	h_1^{g_1}h_2^{g_2} = \left(h_1^{x}h_2\right)^{g_2} \in \left(h_1^{x}H\right)^{g_2}\subseteq \left(h_1^{x}\right)^G = h_1^{G} \subseteq \cup_{g\in G}H^G.\]
	In both cases, the product lies in $\cup_{g\in G}H^g$, showing that this set is closed under multiplication.  Since it is also clearly closed under taking inverses, it is a subgroup of $G$. Therefore, $N=\cup_{g\in G}H^g$, as desired. Moreover, since $H$ is a nontrivial proper subgroup of $G$, the subset $\Delta_H(G) = G\setminus \cup_{g\in G}H^g$ is nonempty. Hence, $G\setminus N$ is nonempty, and since $N$ contains $H$, we obtain $1<N<G$. 
\end{proof}

\begin{lem}\label{lemma8}
Let $G$ be a finite group and let $H$ be a subgroup of $G$.	Suppose that the pair $(G,H)$ satisfies $\F$. Let $N$ be the normal closure of $H$ in $G$. Then $(G,N)$ is a Camina pair. 
\end{lem}
\begin{proof} If $H\unlhd G$, then $(G,H)$ is a Camina pair by definition. So, we may assume that $H$ is not normal in $G$. By Lemma \ref{lemma11}, $N$ is a nontrivial proper normal subgroup of $G$. Since $H\subseteq N\unlhd G$, it follows that  for every $x \in G\setminus N$, $x^t\not\in H$ for all $t\in G$.  By condition $\F$, for each $t\in G$, we have $x^tH\subseteq (x^t)^G = x^G$. Therefore, $$x^GH = \cup_{t\in G}x^tH\subseteq x^G,$$ so $x^GH = x^G$. Now let $K=x^G$. Then  $KH=K$, meaning that for every element $h\in H$, we have $Kh=K$. Consequently, for all $g\in G$,
	\begin{align*}
		Kh^g = K^gh^g = (Kh)^g = K^g = K.
	\end{align*}
Thus, $Kh^g = K$ for all $g\in G, h\in H$. Since every $n\in N$ is of the form  $h^g$ for some $h\in H$ and $g\in G$ by Lemma \ref{lemma11}, it follows that  $Kn = K$ for all $n\in N$, so  $KN = K$. Therefore, $xN \subseteq K = x^G$ for every $x\in G\setminus N$, and hence $(G,N)$ is a Camina pair.  
\end{proof}

\begin{lem}\label{lemma12}
Let $G$ be a finite group and let $H$ be a subgroup of $G$.	Suppose that the pair $(G,H)$ satisfies condition $\F$ and that $H$ is not normal in $G$. Let $N$ be the normal closure of $H$ in $G$. Then $N$ is nilpotent. Moreover, either $G$ is a Frobenius group with Frobenius kernel $N$ or $N$ is a $p$-group. 
\end{lem}
\begin{proof}
It follows from Lemma~ \ref{lemma8} that $(G,N)$ is a Camina pair. By \cite[Theorem 2]{Camina}, one of the following holds: 
	\begin{enumerate}
\item	$G$ is a Frobenius group with Frobenius kernel $N$; 
\item  $N$ is a $p$-group; 
\item  $G/N$ is a $p$-group, where $p$ is a prime. 
\end{enumerate} 

In the first two cases, $N$ is nilpotent (in the first case, this follows from Thompson's Theorem \cite[Theorem 1]{Thompson}). It remains to rule out the third possibility. Suppose, for contradiction, that  $G/N$ is a $p$-group but neither of the first two cases holds. That is, $N$ is not a $p$-group and $G$ is not a Frobenius group with Frobenius kernel $N$.  
	
Since $(G,N)$ is a Camina pair and $G/N$ is a $p$-group, it follows from \cite[Theorem 4.13]{Lewis} that  $G$ has a normal $p$-complement $M$. Since $G/N$ is a $p$-group,  we must have $M\subseteq N$. Furthermore, as $H$ is not normal in $G$, it follows that $M<H$. By Lemma \ref{lemma6}, $(G/M, H/M)$ also satisfies condition $\F$. But $G/M$ is a $p$-group, so  by Lemma \ref{lemma9}, $H/M\unlhd G/M$, which implies that $H\unlhd G$, a contradiction.
\end{proof}

\begin{lem}\label{lemma13}
Let $G$ be a finite group and let $H$ be a subgroup of $G$.	Suppose $(G,H)$ satisfies condition $\F$ and that $H$ is not normal in $G$. Let $N$ be the normal closure of $H$ in $G$. Then $H\unlhd N$ and $N/H$ is abelian. 
\end{lem}
\begin{proof}
	Since $N$ is nilpotent by Lemma \ref{lemma12}, its derived subgroup $N'$ is a proper characteristic subgroup of $N$. Since $N$ is normal in $G$, $N'\unlhd G$. Moreover, since $N$ is the smallest normal subgroup of $G$ containing $H$, we must have $N'<H$. Therefore, $H\unlhd N$ and $N/H$ is abelian.   
\end{proof}

Recall that for a normal subgroup $N$ of a finite group $G$, we define $\Irr(G\mid N)=\Irr(G)\setminus \Irr(G/N)$, that is, $\Irr(G\mid N)$ is the set of irreducible characters of $G$ whose kernels do not contain $N$. Since $N=\la H^g:g\in G\ra$, we see that $\Irr(G\mid N)=\Irr(G\mid H).$

In \cite[Theorem 2.1]{Ren-Lu-Li}, the authors showed that the pair $(G,H)$ satisfies condition $\CI$ if and only if for every $\chi\in\Irr(G\mid N)$, $$(1_H)^G(\chi-[\chi_H,1_H]1_G)=a(\chi)\chi,$$ where $[\chi_H,1_H]$ denotes the inner product of $\chi_H$ with the trivial character $1_H$, $a(\chi)$ is a positive integer and $N$ is the normal closure of $H$ in $G$. Furthermore, if $H\unlhd G$, then $H=N$ and the latter condition simplifies to: for every $\chi\in \Irr(G\mid N)$, $$(1_H)^G\chi=a(\chi)\chi,$$ where $a(\chi)$ is some positive integer. Recall that $(1_H)^G\chi=(\chi_H)^G$ for any $\chi\in\Irr(G).$

It turns out that if the pair $(G,H)$ satisfies the last condition above, then $H\unlhd G$. This is the content of Theorem 3.1 in \cite{Ren-Lu-Li} which we reproduce below.

\begin{lem}\label{lem:RLL} Let $H$ be a subgroup of a finite group $G$ with $1<H<G$. Suppose that for every $\chi\in \Irr(G\mid N)$, $(1_H)^G\chi=a(\chi)\chi,$ where $a(\chi)$ is some positive integer. Then $H\unlhd G$.
\end{lem}

We will need the following consequence of Lemma \ref{lem:RLL}. 

\begin{cor}\label{cor:RLL}
Let $H$ be a  subgroup of a finite group $G$. Suppose that the pair $(G,H)$ satisfies condition $\CI$. If $[\chi_H,1_H]=0$ for every $\chi\in\Irr(G\mid H)$, then $H\unlhd G.$
\end{cor}

\begin{proof}
This is Corollary 3.2 in \cite{Ren-Lu-Li}.
\end{proof}

\section{Proof of Theorem \ref{thm1}} \label{sec:s3}

Let $G$ be a finite group and let $H$ be a nontrivial proper subgroup of $G$. Let $N$ be the normal closure of $H$ in $G$. In the first half of this section, we aim to prove the necessary condition in Theorem \ref{thm1}, namely that if the pair $(G,H)$ satisfies condition $\CI$, then it also satisfies condition $\F$. Our approach is based on analyzing the two possible cases for the restriction $\chi_H$ of characters $\chi\in\Irr(G\mid H)=\Irr(G\mid N)$: either $[\chi_H,1_H]=0$ or $[\chi_H,1_H]\neq 0$. Note that, by  Corollary \ref{cor:RLL}, if  $(G,H)$ satisfies condition $\CI$ and $[\chi_H,1_H]=0$ for every $\chi\in \Irr(G\mid N)$, then  $H\unlhd G$. Hence, assuming $H$ is not normal in $G$, there must exist some  $\chi\in \Irr(G\mid H)$ such that $[\chi_H,1_H]\neq 0$.

\begin{prop}\label{prop6}
Let $G$ be a finite group and let $H$ be a subgroup of $G$.	Suppose that the pair $(G,H)$ satisfies condition $\CI$ and that $H$ is not normal in $G$. Let $N$ be the normal closure of $H$ in $G$. Assume $\chi\in \Irr(G\mid N)$ such that $[\chi_H,1_H] \neq 0$. Then $\chi(xh) = \chi(x)$ for all $x\in N\setminus H$, $h\in H$.
\end{prop}
\begin{proof}
	First, we prove that all irreducible constituents of $\chi_N$ are linear characters.  
	By Clifford's Theorem \cite[Theorem 6.2]{Isaacsbook}, we can write
	\begin{align}\label{5}
		\chi_N = f(\theta_1 +  \ldots + \theta_s  + \theta_{s+1} + \ldots + \theta_t)
	\end{align}
	where $\theta_1,\dots,\theta_t$ are distinct irreducible characters of $N$ that are conjugate under the action of $G$.  Since $[\chi_H,1_H]\neq 0$, we may assume without loss of generality that the trivial character $1_H$ lies under $\theta_1,\ldots, \theta_s$ for some integer $s$ with $1\leq s\leq t$.
	By Lemma \ref{Hnormal}, we know that $H\unlhd N$. Applying Clifford's Theorem again to each $\theta_i$ with $1\leq i\leq s$, the restrictions of all these characters $\theta_1,\ldots,\theta_s$ to $H$ are multiple of $1_H$. Consequently, for $1\leq i\leq s$, $\theta_i$ is an irreducible character of the quotient group $N/H$, which is abelian by Lemma \ref{Hnormal}. Hence  these $\theta_i$ are all linear characters of $N$. Finally, since $\theta_1,\dots,\theta_t$ are conjugate in $G$, they all have the same degree. Therefore, all $\theta_i$, for $1\leq i\leq t$,  are linear characters, as required. 
	
	Next, observe that the restrictions of $\theta_1,\ldots,\theta_s$ to $H$  are all equal to $1_H$. 
	Since $\chi\notin \Irr(G/N)$, where $N$ is the smallest normal subgroup of $G$ containing $H$, it follows that $H\nsubseteq \ker(\chi)$. Thus $\chi_H\neq \chi(1)1_H$, meaning that $1_H$ is not the unique irreducible constituent of $\chi_H$. Therefore, we can choose a constituent $\mu_1$ of $\chi_H$ different from $1_H$.  Assume that $\theta_{s+1},\ldots,\theta_{l_1}$ 
	(for some $s+1\leq l_1\leq t$) are all the irreducible constituents of $\chi_N$ lying over $\mu_1$. 
	We will show that $\left\{\theta_{s+1},\ldots,\theta_{l_1}\right\} = \Irr(N\mid \mu_1)$ and moreover, 
	\begin{align*}
		\mu_1^N = \theta_{s+1}+\ldots  + \theta_{l_1}.
	\end{align*}
	Take an arbitrary   $\theta\in \Irr(N\mid \mu_1)$.  Since $(G,H)$ satisfies condition $\CI$ and $\mu_1\neq 1_H$,  $$\mu_1^G  = (\mu_1^N)^G = e\chi$$ for some positive integer $e$. This implies that $\chi$ is the unique irreducible constituent of $\theta^G$.  Consequently,  $\theta$ must be an irreducible constituent of $\chi_N$ lying over $\mu_1$,  so $\theta\in \left\{\theta_{s+1},\ldots,\theta_{l_1}\right\}$. Conversely, let $\varphi\in \left\{\theta_{s+1},\ldots,\theta_{l_1}\right\}$. Then $ \varphi \in \Irr(N)$ lies over $\mu_1$, so $\varphi\in \Irr(N\mid \mu_1)$. This establishes the equality  $\left\{\theta_{s+1},\ldots,\theta_{l_1}\right\} = \Irr(N\mid \mu_1)$. Furthermore, since $\theta_{s+1},\ldots,\theta_{l_1}$ are linear characters of $N$, $(\theta_{i})_H = \mu_1$ for each $i=s+1,\ldots,l_1$, and it implies that
	\begin{align*}
		[\mu_1^N,\theta_i] = [\mu_1,(\theta_i)_H] = 1.
	\end{align*}
	Therefore, $\mu_1^N = \theta_{s+1} + \ldots + \theta_{l_1}$ as claimed. 
	
	We now repeat the same argument for the remaining nontrivial irreducible constituents of $\chi_H$.   Suppose $\left\{1_H,\mu_1,\mu_2,\ldots,\mu_n\right\}$ is the set of all  irreducible constituents of $\chi_H$, where  $1_H\neq \mu_j\in\Irr(H)$ for $1\leq j\leq n$.	
	Take any $\mu_j$ for $j=2,\dots,n $.  By reindexing the set $\{\theta_1,\dots,\theta_t\}$, we may assume that $\theta_{l_{j-1}+1},\ldots,\theta_{l_j}$ are the irreducible constituents of $\chi_N$ lying over $\mu_j$, where $l_j\leq t$ and  $l_n=t$. Applying the same argument used previously for $\mu_1$, we obtain: $$\left\{\theta_{l_{j-1}+1},\ldots,\theta_{l_j}\right\} = \Irr(N\mid \mu_j)\text{ and }\mu_j^N = \theta_{l_{j-1}+1} + \ldots + \theta_{l_j}.$$  
	To finish the proof of the proposition, we observe that the subset $\left\{\theta_{s+1},\ldots,\theta_{t}\right\}$ of the irreducible constituents of $\chi_N$ in Equation \eqref{5} is a disjoint union
	$$\left\{\theta_{s+1},\ldots,\theta_{t}\right\} = \left\{\theta_{s+1},\ldots,\theta_{l_1}\right\}\cup \left\{\theta_{l_1+1},\ldots,\theta_{l_2}\right\}\cup\ldots\cup\left\{\theta_{l_{n-1}+1},\ldots,\theta_{l_n}\right\}$$ which is equal to
	$$ \Irr(N\mid \mu_1)\cup\Irr(N\mid \mu_2)\cup\ldots\cup\Irr(N\mid \mu_n).$$
	
	Furthermore, $\chi_N$ in Equation \eqref{5} can be expressed as follows
	\begin{align}\label{8}
		\chi_N &=f (\theta_1+\ldots+\theta_s+\theta_{s+1}+\ldots+\theta_{l_1}+\theta_{l_1+1}+\ldots+\theta_{l_2}+\ldots+\theta_{l_{n-1}+1}+\ldots+\theta_{l_n}) \nonumber\\
		&= f(\theta_1+\ldots+\theta_s + \mu_1^N + \mu_2^N+\ldots +\mu_n^N).
	\end{align}

	Finally, as $\theta_1,\ldots,\theta_s$ are characters of the quotient $N/H$, for any $x \in N\setminus H$ and  $h\in H$, we have $\theta_i(xh) = \theta_i(x)$ for all $i = 1,\ldots,s$. Furthermore, as $H\unlhd N$, $x^N$ and $(xh)^N$ intersect $H$ trivially, and thus $\mu_j^N(x) = 0 = \mu_j^N(xh)$ for all $j=1,\ldots,n$. Substituting into  Equation \eqref{8}, we obtain
	\begin{align*}
		\chi(xh) = f(\theta_1(xh) + \ldots+\theta_s(xh)) = f(\theta_1(x) + \ldots + \theta_s(x)) = \chi(x).
	\end{align*} This shows that $\chi$ is constant on the nontrivial left cosets of $H$ in $N$, completing the proof of the proposition.
\end{proof}

\begin{prop}\label{prop7}
Let $G$ be a finite group and let $H$ be a subgroup of $G$.	Suppose that the pair $(G,H)$ satisfies condition $\CI$ and that $H$ is not normal in $G$. Let $N$ be the normal closure of $H$ in $G$. Assume $\chi \in \Irr(G\mid N)$ such that $[\chi_H,1_H]=0$. Then $\chi(xh) = \chi(x)=0$ for all $x\in N\setminus H$ and $h\in H$.
\end{prop}
\begin{proof}
	First, by applying Clifford's Theorem with $N\unlhd G$, we obtain 
	\begin{align}\label{9}
		\chi_N = f(\theta_1+\theta_2+\ldots + \theta_t),
	\end{align}
	where $\theta_i$ for $i=1,2,\ldots,t$ are distinct irreducible characters of $N$, which are conjugate in $G$. Since $[\chi_H,1_H]=0$, we can find a nontrivial irreducible constituent $\mu_1$ of $\chi_H$.
	We may assume that $\theta_1,\ldots,\theta_{l_1}$, for some $l_1\leq t$, are precisely the irreducible constituents of $\chi_N$ lying over $\mu_1$. Analogous to the argument in Proposition \ref{prop6}, we claim $$\Irr(N\mid \mu_1) = \left\{\theta_1,\ldots,\theta_{l_1}\right\}\text{ and }
		\mu_1^N = a_1(\theta_1+\ldots+\theta_{l_1})$$
	for some positive integer $a_1$. To verify this, take a character $\varphi \in \left\{\theta_1,\ldots,\theta_{l_1}\right\}$. Then $\varphi$ is an irreducible character of $N$ lying over $\mu_1$, so $\varphi\in \Irr(N\mid \mu_1)$. Conversely, let $\theta\in \Irr(N\mid \mu_1)$. Since $(G,H)$ satisfies condition $\CI$ and $\mu_1\neq 1_H$, $\mu_1^G = (\mu_1^N)^G = e\chi$ for some positive integer $e$. Thus $\chi$ is the unique irreducible constituent of $\theta^G$, which implies that $\theta$ is an irreducible constituent of $\chi_N$. As $\theta$ lies over $\mu_1$, it must be among $ \left\{\theta_1,\ldots,\theta_{l_1}\right\}$. This shows that $\left\{\theta_1,\ldots,\theta_{l_1}\right\} = \Irr(N\mid \mu_1)$, as claimed.
	
	Next, since $H\unlhd N$ by Lemma \ref{Hnormal}, we assume that for some elements $n_1,\ldots,n_s\in N$,  the characters $\mu_1,\mu_1^{n_1},\ldots,\mu_1^{n_s}\in \Irr(H)$ are all the distinct $N$-conjugates of $\mu_1$.
	Because $\theta_i$ lies over $\mu_1$ for all $i=1,\ldots,l_1$, applying Clifford's Theorem for $H\unlhd N$, we obtain
	$$(\theta_i)_H = b_i(\mu_1+\mu_1^{n_1} + \ldots + \mu_1^{n_s}), $$
	for some positive integer $b_i$ and for all $i = 1,\ldots, l_1$. Since $\theta_1,\ldots,\theta_{l_1}$ have the same degree and since $\mu_1,\mu_1^{n_1},\ldots,\mu_1^{n_s}$ have the same degree,  it follows that
	$$	  b_i(s+1)\mu_1(1)= \theta_i(1)= \theta_1(1) = b_1 (s+1)\mu_1(1),$$  which implies that $b_1=b_2=\dots=b_{l_1}$. Denote this common value by $a_1$, we conclude that $a_1=[(\theta_i)_H,\mu_1] = [\theta_i,\mu_1^N]$ for all $i  = 1,\ldots,l_1$. Therefore, since $\Irr(N\mid\mu_1) = \left\{\theta_1,\ldots,\theta_{l_1}\right\}$,  it follows that 
	\[ \mu_1^N = a_1(\theta_1+\ldots+\theta_{l_1}) \text{ or equivalently } \theta_1 + \ldots + \theta_{l_1} = \dfrac{1}{a_1} \mu_1^N. \]
	
	We now proceed, as in the proof of  Proposition \ref{prop6}, to partition   the set $\left\{\theta_1,\theta_2,\ldots,\theta_t\right\}$ of irreducible constituents  of $\chi_N$. To that end, suppose that  $\mu_2,\ldots,\mu_n$ together with $\mu_1$ are all the irreducible constituents of $\chi_H$ that are not conjugate in $N$. We claim that  the sets $\Irr(N\mid \mu_1),\Irr(N\mid\mu_2),\ldots,\Irr(N\mid \mu_n)$ form a partition of $\left\{\theta_1,\ldots,\theta_t\right\}$. To see that these sets are pairwise disjoint, suppose otherwise. Then for some $j\neq k$, $\Irr(N\mid\mu_j)$ and $\Irr(N\mid \mu_k)$ would share a common element $\lambda\in\Irr(N)$, implying that both $\mu_j$ and $\mu_k$ appear in $\lambda_H$. However, since $H\unlhd N$, Clifford's Theorem would imply that $\mu_j$ and $\mu_k$ are conjugate in $N$, contradicting the assumption. 
	
	Now, since $\mu_j$ for $j=2,\ldots,n$ is a constituent of $\chi_H$, it must  lie under some $\theta_i\in \{\theta_1,\dots,\theta_t\}$. We may assume  that each $\mu_j$ lies under $\theta_{l_{j-1}+1},\ldots,\theta_{l_{j}}$ for $j=2,\ldots,n$ with $l_0=0$ and ${l_n} = t$. Repeating the previous argument, we conclude that for all $j=2,\ldots,n$, $$\left\{\theta_{l_{j-1}+1},\ldots,\theta_{l_j}\right\} = \Irr(N\mid \mu_j)\text{ and }
	\theta_{l_{j-1}+1} + \ldots + \theta_{l_j} = \dfrac{1}{a_j}\mu_j^N,$$
	for some positive integer $a_j$. Therefore, we obtain the desired partition $$\Irr(N\mid \mu_1)\cup \Irr(N\mid \mu_2)\cup \ldots\cup\Irr(N\mid\mu_n)=\left\{\theta_1,\ldots,\theta_t\right\}.$$ From Equation \eqref{9}, 
	 	\begin{align}\label{1}
		\chi_N &= f (\theta_1+ \ldots +\theta_{l_1} + \theta_{l_1+1}+\ldots + \theta_{l_2} + \ldots + \theta_{l_{n-1}+1} + \ldots +\theta_{l_n}) \nonumber\\
		& = f\left(\dfrac{1}{a_1}\mu_1^N + \dfrac{1}{a_2}\mu_2^N  + \ldots + \dfrac{1}{a_n}\mu_n^N\right).
	\end{align} 
	Finally, observe that for all  $x\in N\setminus H$ and  $h\in H$, since $H\unlhd N$, $x^N$ and $(xh)^N$ intersect $H$ trivially. Therefore, $\mu_j^N(x) = 0 = \mu_j^N(xh)$  for all $j=1,\ldots,n$. Substituting into Equation \eqref{1}, we conclude that $\chi(x) = 0=\chi(xh)$, which completes the proof. 
\end{proof}

We are now ready to prove the first implication of Theorem \ref{thm1}.
\begin{prop}\label{propn15}
	Let $G$ be a finite group and let $H$ be a subgroup of $G$. Suppose $H$ is not normal in $G$. If $(G,H)$ satisfies condition $\CI$, then $(G,H)$ satisfies condition $\F$.
\end{prop}
\begin{proof}
	Let $N$ denote the normal closure of $H$ in $G$. By Lemma \ref{lem:Cam}, $(G,N)$ is a Camina pair. That is, for every $x\in G\setminus N$, $x$ is conjugate to $xn$ for all $n\in N$. In particular, for all $x\in G\setminus N$ and $h\in H\subseteq N$, $x$ is conjugate to $xh$, as required for condition $\F$. It remains to consider  the case when $x\in N\setminus H$. We claim that  for all $x\in N\setminus H$, $x$ is conjugate  to $xh$ for all $h\in H$. This is equivalent to the fact that, for all $x\in N\setminus H$, $\chi(x) = \chi(xh)$ for all $h\in H$ and $\chi\in \Irr(G)$. Indeed, for any $\chi\in \Irr(G/N)$, we have $\chi(x) = \chi(xh) = \chi(1)$ for all $x \in N, h\in H$ as $N$ is contained in the kernel of $\chi$.  For  $\chi \in \Irr(G\mid N)$, the equality $\chi(x) = \chi(xh)$ for all $x\in N\setminus H$, $h\in H$ follows from  Propositions \ref{prop6} and \ref{prop7}. Thus, all irreducible characters of $G$ are constant on the cosets $xH$ for $x\in N\setminus H$, so $xH\subseteq x^G$ for all $x\in N\setminus H.$ Hence $(G,H)$ satisfies condition $\F$, completing the proof.
\end{proof}

Next, we assume $(G,H)$ satisfies condition $\F$ and prove the converse implication in Theorem \ref{thm1}.

\begin{prop}\label{propn17}
	Let $G$ be a finite group with a nontrivial proper subgroup $H$. Suppose that $H$ is not normal in $G$. If $(G,H)$ satisfies condition $\F$, then $(G,H)$ satisfies condition $\CI$.
\end{prop}
\begin{proof}
	Let $\lambda$ be a  nontrivial irreducible character of $H$. 
	Let $N$ be the normal closure of $H$ in $G$. We first claim that every constituent of  $\lambda^N$ is nontrivial. Suppose, for contradiction, that  $1_N$ is a constituent of $\lambda^N$.  Then:   $$0\neq [\lambda^N,1_N] = [\lambda,1_H]=0,$$ a contradiction. Thus, all constituents of $\lambda^N$ are nontrivial. 
	
	Let $\theta$ be a constituent of $\lambda^N$ and let $\rho$ be an irreducible representation of $N$ affording $\theta$.  By Lemma \ref{lemma13}, $H\unlhd N$ and applying Clifford's Theorem, we have 
\[		\theta_H = a(\lambda_1 + \lambda_2+\ldots + \lambda_t),\]
	where $\lambda_1 = \lambda$ and $\lambda_1,\dots,\lambda_t$ are all the distinct conjugates of $\lambda$ in $N$. Since $\lambda$ is nontrivial,  all the $\lambda_i$ are also nontrivial.
	
	 Now, let $\widehat{H}=\sum_{h\in H}h \in \mathbb{C}[N]$ be the sum of all elements in $H$. Since $H\unlhd N$, $H$ is a union of some conjugacy classes of $N$, so $\widehat{H}$ lies in the center of $\CC[N]$. Hence, by the Schur's Lemma (\cite[Lemma 2.25]{Isaacsbook}), $\rho(\widehat{H})=\alpha I$  for some $\alpha \in\CC$, where $I$ denotes the identity matrix of size $\theta(1)$. We  compute: 
	\begin{align*}
		\mathrm{trace}(\rho(\widehat{H})) &=  \sum_{h\in H}\mathrm{trace}(\rho(h)) = \sum_{h\in H}\theta_H(h) = |H|[\theta_H,1_H] \\
		&=a|H|[\lambda_1+\lambda_2+\ldots+\lambda_t,1_H] = 0,
	\end{align*} since each $\lambda_i$ is nontrivial. Hence $\alpha=0$ so $\rho(\widehat{H}) = 0$.   
	
	By Lemma \ref{lemma8}, $(G,N)$ is a Camina pair. By \cite[Lemma 4.1 (7)]{Lewis}, every nontrivial irreducible character of $N$ induces homogeneously to $G$.  
	Therefore, every $\theta\in \Irr(N\mid \lambda)$ induces homogeneously to $G$. 
	We will show that all such $\theta$ induce to the same  character $\chi$ for some $\chi\in \Irr(G)$, so that $\lambda^G=e\chi$  for some integer $e>0$.
	
	By Lemma \ref{lemma11}, we have $N = \cup_{g\in G}H^g$. Hence, we can write $N=H\cup h_1^{g_1}H\cup \ldots \cup h_n^{g_n}H$ for some $h_1,\ldots,h_n\in H$ and $g_1,\ldots,g_n\in G$ with $h_i^{g_i}\not\in H$.
	Let $\chi\in \Irr(G)$ be an irreducible constituent of $\lambda^G$. We will show that $\chi$ is the unique constituent of  $\theta^G$ for all  $\theta\in \Irr(N\mid \lambda)$. To do that, we will show $[\chi,\theta^G] $ is nonzero for all $\theta\in \Irr(N\mid\lambda)$. Note that $[\chi,\theta^G] = [\chi_N,\theta]$. We compute
	
	$$\begin{array}{lll}
		|N|[\chi_N,\theta] &=& \sum_{n\in N}\chi_N(n)\overline{\theta(n)}\\
		& =&  \sum_{h\in H}\chi_H(h)\overline{\theta_H(h)} +\sum_{j=1}^n \sum_{h\in H}\chi(h_j^{g_j}h)\overline{\theta(h_j^{g_j}h)}. \\
		\end{array} $$
	From condition $\F$, each  $h_j^{g_j}\notin H$ is conjugate to $h_j^{g_j}h$ for all $h\in H$, so $$\chi(h_j^{g_j}h) = \chi(h_j^{g_j}) = \chi(h_j).$$ Thus
\begin{equation}\label{eqn3}
|N|[\chi_N,\theta]= |H|[\chi_H,\theta_H]  +\sum_{j=1}^n \chi(h_j)\sum_{h\in H} \overline{\theta(h_j^{g_j}h)} .
\end{equation}
Finally, since $\rho(\widehat{H})=0$, we obtain
\[\rho(h_j^{g_j}\widehat{H}) = \rho(h_j^{g_j})\rho(\widehat{H}) = \rho(h_j^{g_j})\cdot 0 = 0.\]
This implies that for each $j=1,\ldots,n$, $$0=\mathrm{trace}\left(\rho(h_j^{g_j}\widehat{H})\right) =\sum_{h\in H}\mathrm{trace}\left(\rho(h_j^{g_j}h)\right) = \sum_{h\in H}\theta(h_j^{g_j}h).$$  Therefore, Equation \eqref{eqn3} simplifies to $$|N|[\chi_N,\theta] = |H|[\chi_H,\theta_H]\neq 0$$  because $\chi_H$ and $\theta_H$ share a common irreducible constituent $\lambda$. 
Hence $[\chi,\theta^G]\neq 0$, and since $\theta^G$ is homogeneous, $\chi$ is the only irreducible constituent of $\theta^G.$ As this applies to all $\theta\in\Irr(N\mid \lambda),$ it follows that all of them induce to the same irreducible character $\chi\in\Irr(G),$ so $\lambda^G=e\chi$ for some $e>0$, completing the proof.
\end{proof}

\begin{proof}[\textbf{Proof of Theorem \ref{thm1}}]
If $H$ is normal in $G$, then this is Lemma 2.1 in \cite{Kuisch-vanderWaall}. Assume that $H$ is not normal in $G$. Then the result follows from Propositions \ref{propn15} and \ref{propn17}. 
\end{proof}
\section{Proof of Theorem \ref{thm2}}

In this section, we examine the case where a finite group $G$ with a subgroup $H$ satisfies condition $\Fp$.   Our goal is to establish a series of results analogous to those obtained in earlier sections, where $(G,H)$ satisfies $\CI$ or $\F$.  

\begin{lem}\label{lemma18}
Let $G$ be a finite group and let $H$ be a subgroup of $G$.	Suppose that the pair $(G,H)$ satisfies condition $\Fp$ and that $H$ is not normal in $G$. If $M \unlhd G$ and $H\nsubseteq M$, then $M<H$. Moreover, the pair $(G/M, H/M)$ satisfies condition $\Fp$. 
\end{lem}
\begin{proof}
	Suppose $M\nsubseteq H$ and let $x\in M\setminus H$. Since $(G,H)$ satisfies condition $\Fp$, $xH\subseteq x^G\cup(x^{-1})^G \subseteq M$, where the last containment follows from the normality of $M$ in $G$. Hence, $H\subseteq x^{-1}M = M$, a contradiction. Thus $M\subseteq H$. Since $H\nsubseteq M$, we have $H\neq M$. Hence $M<H$.
	The second claim is obvious, we skip the proof.
\end{proof}

\begin{cor}\label{corol19}
Let $G$ be a finite group and let $H$ be a subgroup of $G$.	Suppose that the pair $(G,H)$ satisfies condition $\Fp$ and that $H$ is not normal in $G$. Then $\z(G) < H<G'$.
\end{cor}
\begin{proof}
	First, suppose that $H\nsubseteq \z(G)$. Then by Lemma \ref{lemma18}, we obtain $\z(G)< H$, and the first containment follows. So we may assume that $H\leq  \z(G)$. But then $H$ is a normal subgroup of $G$. This contradicts with our assumption. Therefore, $\z(G) < H$. Similarly, if $H\nsubseteq G'$, then by Lemma \ref{lemma18}, $G'<H$. This again implies that $H$ is a normal subgroup of $G$, a contradiction. Henceforth, we must have $H < G'$. 
\end{proof}

Let $H$ be a proper subgroup of a finite group $G$. We denote by $\Delta_H(G)$ the set $$\Delta_H(G)= G\setminus \cup_{g\in G}H^g.$$ Elements in $\Delta_H(G)$ are called $H$-derangements of $G$. An easy counting argument shows that $\Delta_H(G)$ is non-empty.

\begin{lem}\label{lemma20}
Let $G$ be a finite group and let $H$ be a subgroup of $G$.	Suppose that the pair $(G,H)$ satisfies condition $\Fp$ and that $H$ is not normal in $G$. Let $x\in \Delta_H(G)$ and $K=x^G$. Let $N$ be the normal closure of $H$ in $G$. Then $KN \subseteq K\cup K^{-1}$. In particular, the pair $(G,N)$ satisfies condition $\Fp$.
\end{lem}

\begin{proof}
	For each $g\in G$, we first claim that $KH^g\subseteq K\cup K^{-1}$ and $K^{-1}H^g \subseteq K\cup K^{-1}$. Since $x\in \Delta_H(G)$, $x^g \notin H$, and by condition $\Fp$,  we have $$x^gH\subseteq (x^g)^{G}\cup \left((x^g)^{-1}\right)^G = K\cup K^{-1}$$ for all $g\in G$. Thus,  $KH\subseteq K\cup K^{-1}$. Hence, $K^gH^g \subseteq K^g\cup (K^{-1})^g$ for every $g\in G$. Since $K$ is a conjugacy class of $G$, $K^g = K$ and so $$KH^g \subseteq K\cup K^{-1} \text{ and }K^{-1}H^g \subseteq K\cup K^{-1}.$$ Now, let $g_1, g_2\in G$. Using the previous results, for $i=1,2$,  $$KH^{g_i} \subseteq K\cup K^{-1}\text{ and }K^{-1}H^{g_i}\subseteq K\cup K^{-1}.$$  Therefore, 
	\[	KH^{g_1}H^{g_2}\subseteq (K\cup K^{-1})H^{g_2} \subseteq KH^{g_2} \cup K^{-1}H^{g_2} \subseteq K\cup K^{-1}.\]
	Since $N = \la H^{g}:g\in G\ra$, every element $n\in N$ can be written as a product of finitely many  elements from various  conjugates  $H^g$, for  $g\in G$. Thus, by iteratively applying the containment above, we get  $Kn \subseteq K\cup K^{-1}$ and hence $KN \subseteq K\cup K^{-1}$. This containment also implies that $N<G$, because if $N=G$, then $1\in G= KG = KN \subseteq K \cup K^{-1}$, hence $x=1$ which is not an $H$-derangement of $G$. As $N$ contains $H$, $N$ is a nontrivial proper normal subgroup of $G$.
	Moreover, for all $x\in G\setminus N \subseteq \Delta_H(G)$, set $K = x^G$, as $KN \subseteq K\cup K^{-1}$, we have  $xN \subseteq K\cup K^{-1}=  x^G \cup (x^{-1})^G$, which is precisely condition $\Fp$ for the pair $(G,N)$. The proof is now complete.
\end{proof}

\begin{cor}\label{corol22}
Let $G$ be a finite group and let $H$ be a subgroup of $G$.	Suppose that the pair $(G,H)$ satisfies condition $\Fp$ and that $H$ is not normal in $G$. Let $N$ be the normal closure of $H$ in $G$. Then $N = \cup_{g\in G}H^g$. 
\end{cor}
\begin{proof}
	Assume, for contradiction,  that $\cup_{g\in G}H^g\subsetneq N$. There exists an element $x \in N\setminus \cup_{g\in G}H^g \subseteq \Delta_H(G)$. Let $K = x^G$. By Lemma \ref{lemma20}, we have $KN\subseteq K\cup K^{-1}$. Since $x\in N$,  $K\subseteq N$, and therefore, $N \subseteq K\cup K^{-1}$. But $1\in N$, so $1\in  K\cup K^{-1}$, implying that  $x = 1$, a contradiction since $x\not\in \cup_{g\in G}H^g$. Hence   $ N= \cup_{g\in G}H^g$ as wanted.
\end{proof}

\begin{lem}\label{lemma23}
Let $G$ be a finite group and let $H$ be a subgroup of $G$.	Suppose that the pair $(G,H)$ satisfies condition $\Fp$ and that $H$ is not normal in $G$. Let $N$ be the normal closure of $H$ in $G$. Then $N$ is solvable and has a normal $p$-complement for some prime $p$ dividing $|G|$.
\end{lem}
\begin{proof}
	Since $H$ is a proper subgroup of $G$,  $\Delta_H(G)=G\setminus \cup_{g\in G}H^g$ is non-empty. By Corollary \ref{corol22}, we have $\Delta_H(G)=G\setminus N$ and so $N<G$. Let $x\in \Delta_H(G)$ and let $K=x^G$. By Lemma \ref{lemma20}, $KN \subseteq K\cup K^{-1}$, and hence $xN\subseteq K\cup K^{-1}$.  By \cite[Theorem 3.3(a)]{Beltran}, $N$ is solvable. Since $G/N$ is nontrivial, we can find $x\in G\setminus N$ such that $x$ is a $p$-element for some prime $p$. By \cite[Theorem 3.3(c)]{Beltran}, $N$ has a normal $p$-complement. 
\end{proof}

\begin{lem}\label{lemma24}
Let $G$ be a finite $p$-group for some prime $p$ and let $H$ be a subgroup of $G$.	If the pair $(G,H)$ satisfies  condition $\Fp$, then $H\unlhd G$. 
\end{lem}
\begin{proof}
	If $H\unlhd G$, then we are done. We may assume that $H$ is not normal in $G$. Then using Lemma~\ref{lemma18} and Corollary~\ref{corol19}, the result follows by applying the same argument as in the proof of Lemma~\ref{lemma9}. 
\end{proof}

\begin{proof}[\textbf{Proof of Theorem \ref{thm2}}]
	By Lemma \ref{lemma20}, the pair $(G,N)$ satisfies condition $\Fp$. Suppose by contradiction that $N$ is not nilpotent.  By \cite[Theorem A]{Akhlaghi-Beltran}, $G/N$ is a $p$-group for some prime $p$, and $G$ is a $p$-nilpotent solvable group. 
	
	We first show that $N$ is a normal $p$-complement of $G$. 
	Let $P$ be a Sylow $p$-subgroup of $G$ and let $Q$ be a normal $p$-complement of $G$, so that $G = Q \rtimes P$. As $G/N$ is a $p$-group, it follows that $Q \unlhd N \unlhd G$. By Lemma \ref{lemma18}, either $H\leq Q$ or $Q\leq H$. If the latter case holds, then $(G/Q,H/Q)$ satisfies condition $\Fp$, and since $G/Q$ is a $p$-group, Lemma \ref{lemma24} implies that  $H$ is normal in $G$, contradicting our hypothesis.  Therefore, $H \leq Q$ and since $N$ is the normal closure of $H$, we have $ N=Q$. In other words,   $N$ is a normal $p$-complement of $G$. Lemma \ref{lemma23} already ensures  that $N$ is solvable.  
	
	Next, we show that all elements in $G\setminus N$ are $p$-elements. Let $x\in G\setminus N$. Since $G/N$ is a $p$-group, $x^{p^a} \in N$ for some positive integer $a$. Write $x = x_px_{p'}$ where $x_p, x_{p'}\in \la x\ra$ with $x_p$ a $p$-element and $x_{p'}$  a $p'$-element, respectively. If $o(x_p) = p^c$ for some positive integer $c$, then $c\geq a$ and so $x^{p^c}\in N$. Thus $x^{p^c} = x_{p'}^{p^c} \in N$, hence $x_{p'}\in N$. Since $x_p\notin N$ and $(G,N)$ satisfies condition $\Fp$, we have  
	\[xN = x_px_{p'}N = x_pN\subseteq x_p^{G} \cup (x_p^{-1})^G.\]
	So, all elements in $xN$ are $p$-elements. It follows that $x$ itself is a $p$-element. Hence, all elements in $G \setminus N$ are $p$-elements. 
	
	We now prove that $G$ is a Frobenius group with Frobenius kernel $N$, so that $N$ is nilpotent by  Thompson's Theorem, yielding a contradiction.  Recall that $N$ is a non-nilpotent normal $p$-complement of $G$.  Since $G=PN$, with $N$ a normal $p$-complement and $P\cap N=1$, it suffices (see \cite[Problem 7.1]{Isaacsbook}) to show that  $\mathbf{C}_G(n) \subseteq N$ for every $1\neq n\in N.$
	Suppose, for contradiction, that there exists $x\in \mathbf{C}_G(n) \setminus N$ for some $1\neq n\in N$. Then $x$ is a $p$-element and $xn=nx$. Since $n$ is a $p'$-element, the order of $xn$ is $o(x)o(n)$, so $xn$ is not a $p$-element. But since $x\in G\setminus N$, condition $\Fp$ implies that $xn\in xN \subseteq x^G \cup (x^{-1})^G$, a set consisting of only $p$-elements. This contradiction shows that $\mathbf{C}_G(n)\subseteq N$, and thus $G$ is a Frobenius group with Frobenius kernel $N$. By Thompson's theorem \cite[Theorem 1]{Thompson}, $N$ is nilpotent, contradicting our assumption. Therefore, $N$ must be nilpotent, completing the proof.   
\end{proof}

\section{Orders of elements in cosets}\label{sec5}

Let $G$ be a finite group and let $p$ be a prime. An element $g\in G$ is called $p$-singular or $p$-regular (or a $p'$-element) if the order of $g$ is divisible by $p$ or not divisible by $p$, respectively.
For brevity, we say that the pair $(G,H)$ satisfies condition $\OO$ if $H<G$ and for every element $x\in G\setminus H$ of odd order, all the elements in the coset $xH$ have odd order. Let the pair $(G,H)$ satisfy condition $\OO$. Observe that if there exists an element  $x\in G\setminus H$ of odd order such that $x$ normalizes $H$, then $H$ is solvable by applying Theorem B (c) in \cite{GN}.  However, it is possible that $G\setminus H$ might not possess any element of odd order. We handle this case in the following lemma. Recall that $O^p(G)$ is the smallest normal subgroup of $G$ such that $G/O^p(G)$ is a finite $p$-group. Furthermore, $O^p(O^p(G))=O^p(G)$.

\begin{lem}\label{lem:p-singular}
Let $G$ be a finite group, let $H$ be a  proper subgroup of $G$ and let $p$ be a prime. Then  

\begin{enumerate}
\item[$(1)$] All elements  in $G\setminus H$ are $p$-singular if and only if  $O^p(H)\unlhd G$ and $G/O^p(H)$ is a $p$-group. 
\item[$(2)$] In particular, if  part (1) holds, then $H$ is subnormal in $G$, and every element in $G\setminus O^p(H)$ is $p$-singular. 
\end{enumerate}
\end{lem}

\begin{proof} (1) Suppose that all elements in $G \setminus H$ are $p$-singular. First, we claim that  $|G:H|$ is a power of $p$. Suppose not, and let $q\neq p$ be a prime dividing $|G:H|$. Then a Sylow  $q$-subgroup $Q$ of $G$ is not contained in  $H$, so there exists a $q$-element outside $H$, contradicting the assumption that all elements in $G\setminus H$ are $p$-singular. Hence $|G:H|$ is a $p$-power. 
	
Now, we will show that $O^p(H) \unlhd G$ and $G/O^p(H)$ is a $p$-group.  Set $K = O^p(H)$, so $H/K$ is a $p$-group. If $K$ is a $p$-group, then $H$ is a $p$-group and hence $K=1$. In this case, it follows from the preceding paragraph that  $G$ is a  $p$-group, and the conclusion follows. Assume now that $K$ is not a $p$-group, and thus $G$ is not a $p$-group. Let $x\in G$ be an arbitrary $p$-regular element.  Since all elements in $G\setminus H$ are $p$-singular, $x \in H$. Moreover, since $H/K$ is a $p$-group, it follows that $x\in K$. Hence, all $p$-regular elements of $G$ lie in $K$. 
	
Let $K_0$ be the subgroup of $G$ generated by all $p$-regular elements of $G$. Then $K_0\unlhd G$ because the set of $p$-regular elements is closed under conjugation. We will show that $K_0 = K$, and so $K\unlhd G$.  From the previous paragraph, $K_0\subseteq K$. Now consider an arbitrary element $hK_0 \in H/K_0$, and write   $h=h_ph_{p'}$ where $h_p,h_{p'} $ are the $p$-part and $p'$-part of $h$, respectively. Then $h_{p'}\in K_0$, so $hK_0 = h_pK_0$ has  $p$-power order in $H/K_0$.  Thus all elements of $H/K_0$ are $p$-elements, and so $H/K_0$ is a $p$-group. Hence  $K\subseteq K_0$, and we conclude that $K= K_0\unlhd G$. Moreover, $|G:H|$ is a $p$-power and $H/K$ is a $p$-group, which implies that $G/K$ is a $p$-group.

Conversely,  suppose $O^p(H)\unlhd G$ and $G/O^p(H)$ is a $p$-group. Then every element $g\in G\setminus H$ maps to a nontrivial element in the $p$-group $G/O^p(H)$, and hence must be $p$-singular. Thus, $G\setminus H$ contains only $p$-singular elements.

(2)  Suppose part (1) holds. Then $O^p(H) \unlhd G$ and  $G/O^p(H)$ is a $p$-group. It follows that $H/O^p(H)$ is subnormal in $G/O^p(H)$, so $H$ is subnormal in $G$. Now, every element $g$ in $G\setminus O^p(H)$ maps to a nontrivial element in the $p$-group $G/O^p(H)$, hence $g$ must be a $p$-singular element. 
\end{proof}

\begin{lem}\label{lem:subgroup}
Let $G$ be a finite group and let $H$ be a subgroup of $G$. If the pair $(G,H)$ satisfies condition $\OO$, then the pair $(G,O^2(H))$ satisfies condition $\OO$. 
 \end{lem}

\begin{proof}
Assume that $(G,H)$ satisfies condition $\OO$. Let $g\in G\setminus O^2(H)$ be an element of odd order. We want to show that every element in the coset $gO^2(H)$ has odd order. If $g\notin H$, then all elements in $gH$ have odd order by condition $\OO$. Hence, all elements in $gO^2(H)$ have odd order.  If $g\in H$, then since $g\notin O^2(H)$, $gO^2(H)$ is a nontrivial odd order element in $H/O^2(H)$, which is a contradiction as $H/O^2(H)$ is a $2$-group. Therefore, such an element does not exist and thus $(G,O^2(H))$ satisfies condition $\OO$. 
\end{proof}

We make the following remark. 

\begin{rem}\label{lem:directproduct}
Let $S$ be a finite nonabelian simple group and let $N=S_1\times S_2\times\cdots\times S_k$, where each $S_i\cong S$ and $k\ge 1.$ If $K\unlhd N$ and $K<N$, then $N/K$ is not solvable (so it is not a $2$-group).
\end{rem}

The next theorem contains most of the assertions in Theorem \ref{th:odd-order-coset}.
\begin{thm}\label{th:odd-order}
Let $G$ be a finite group and let $H$ be a proper subgroup of $G$. Suppose that the pair $(G,H)$ satisfies condition $\OO$. Then either $O^2(H)\unlhd G$ or $H$ is solvable. 
\end{thm}

\begin{proof}
Suppose that the pair $(G,H)$ is a counterexample to the theorem with $|G|+|H|$ minimal. Then $(G,H)$ satisfies condition $\OO$ but $O^2(H)$ is not normal in $G$ and $H$ is nonsolvable. By Lemma \ref{lem:subgroup}, the pair $(G,O^2(H))$ also satisfies condition $\OO$.

Assume $O^2(H)<H$. Then $|G|+|O^2(H)|<|G|+|H|$, and by the minimality of the counterexample, the pair $(G,O^2(H))$ cannot be a counterexample. Hence,  either $O^2(O^2(H))\unlhd G$ or $O^2(H)$ is solvable.  But since $O^2(O^2(H))=O^2(H)$, and $H$ is solvable if and only if $O^2(H)$ is solvable, both possibilities contradict the assumption that $(G,H)$ is a counterexample. Thus, we may assume $H=O^2(H)$. By Lemma \ref{lem:p-singular}, $G\setminus H$ contains at least one  element of odd order. 

\smallskip
\textbf{Claim $1$.} 
Let $K$ be a proper subgroup of $G$ such that $K\not\leq H$. Then $O^2(H\cap K)\unlhd K$ and $K/O^2(H\cap K)$ is a $2$-group or $H\cap K$ is solvable.
\smallskip

Since $H\cap K<K$, the pair $(K,H\cap K)$  satisfies condition $\OO$. Thus $(K,O^2(H\cap K))$ satisfies condition $\OO$ by Lemma \ref{lem:subgroup}. Since $|K|+|O^2(H\cap K)|<|G|+|H|$,  $(K,O^2(H\cap K))$ cannot be a counterexample. Hence, either $O^2(H\cap K)\unlhd K$ or $O^2(H\cap K)$ is solvable. If  $O^2(H\cap K)$ is solvable, then so is $H\cap K$ and we are done. So, assume $O^2(H\cap K)$ is nonsolvable. Then $O^2(H\cap K) \unlhd K$. If $K\setminus O^2(H\cap K)$ has an element of odd order, say $x$, then  every element in the coset $xO^2(H\cap K)\subseteq xH$ has odd order. By  \cite[Theorem B]{GN},  $O^2(H\cap K)$ is solvable, a contradiction. Therefore, $K\setminus O^2(H\cap K)$  has no element of odd order. By Lemma \ref{lem:p-singular}, $O^2(H\cap K) \unlhd K$ and $K/O^2(H\cap K)$ is a $2$-group as claimed.

\smallskip
\textbf{Claim $2$.}  
If $ g \in G \setminus H $ has even order, then every element in the coset $ gH $ also has even order.
\smallskip

This claim follows immediately from condition $ \OO $. If $gH$ has an element of odd order, then all elements in $gH$ also have odd order, and so is $g$. However, $g \in G\setminus H$ is assumed to have even order, a contradiction.

\smallskip
\textbf{Claim $3$.}  
Let $ N \trianglelefteq G $ be a normal subgroup. If $ HN/N $ is a proper subgroup of $ G/N $, then the pair $ (G/N, HN/N) $ satisfies condition $ \OO $.
\smallskip

Assume that $ HN/N $ is a  proper subgroup of $ G/N $. Let $ \overline{G} = G/N $, and for any $ g \in G $, write $ \overline{g} $ for its image in $ \overline{G} $.  Suppose $ \overline{g} \in \overline{G} \setminus \overline{H} $ has odd order. Then there exists $x\in G\setminus H$ of odd order such that $ \overline{g} = \overline{x}$. Since $ x $ has odd order and the pair $ (G,H) $ satisfies condition $ \OO $, it follows that all elements in the coset $ xH $ have odd order.
Hence, the image coset $ \overline{x} \, \overline{H} = \overline{g} \, \overline{H} $ consists entirely of elements of odd order in $ \overline{G} $. Therefore, the pair $ (G/N, HN/N) $ satisfies condition $ \OO $, as required.

\smallskip
\textbf{Claim $4$.}  $ G = \langle H^G \rangle $; that is, the normal closure of $ H $ in $ G $ is  $ G $.
\smallskip

Let $ L = \langle H^G \rangle $ denote the normal closure of $ H $ in $ G $. Recall that $O^2(H)=H$. Suppose, for contradiction, that $ L < G $. Since $ H $ is not normal in $ G $, we have $ H < L < G $. Then the pair $ (L, H) $ satisfies condition $ \OO $. Since $|L|+|H|<|G|+|H|$, and $H$ is nonsolvable, we must have $ H \trianglelefteq L \trianglelefteq G $.

Suppose now that $ L/H $ is not a $ 2 $-group. Then there exists an element $ x \in L \setminus H $ of odd order. Since $ x $ normalizes $ H $, and the coset $ xH $ consists entirely of odd order elements (by condition $ \OO $), it follows from \cite[Theorem~B]{GN} that $ H $ is solvable, a contradiction. Thus, $ L/H $ is a $ 2 $-group. Now observe that $ O^2(L) \unlhd H=O^2(H) $ and $H/O^2(L)$ is a $2$-group, it follows that $H=O^2(L)$. Since $ O^2(L) $ is characteristic in $ L \trianglelefteq G $, we conclude that $ H=O^2(L) \trianglelefteq G $, a contradiction.
Therefore, our assumption that $ L < G $ must be false. We conclude that $ G = \langle H^G \rangle $, as claimed.

\smallskip
\textbf{Claim $5$.}  
The core of $H$ in $G$ is trivial; that is, $H_G = \cap_{g\in G}H^g=1$.

\smallskip
Suppose, for contradiction, that $N = H_G > 1$. Then $ N \trianglelefteq G $, $ N \leq H $, and $ 1 < N < H $. By Claim~3, the pair $(G/N, H/N)$ satisfies condition $ \OO$. Since $H$ is not normal in $ G$,  $H/N$ is not normal in $G/N$, and thus by minimality of the counterexample, $H/N $ must be solvable.
Now, as $H = O^2(H)$ is not normal in $G$, by Lemma \ref{lem:p-singular}, there exists $ x \in G \setminus H $  of odd order. Then $xN \subseteq xH $, and by condition $ \OO$, all elements of $xN $ have odd order. Hence, by \cite[Theorem~B]{GN}, $N$ is solvable.
Since $N \trianglelefteq H $ and both $ N $ and $H/N$ are solvable, it follows that $ H $ is solvable, a contradiction. Therefore, $ H_G = 1$, as wanted.

\smallskip
\textbf{Claim $6$.}  
If $ N $ is a minimal normal subgroup of $ G $, then $G = HN.$
\smallskip

Let $ N$ be a minimal normal subgroup of $G$. Suppose, for contradiction, that $ HN < G $.
It follows that $ HN/N $ is a  proper subgroup of $ G/N $. Thus, either $HN/N \trianglelefteq G/N$ or the quotient $HN/N \cong H / (H \cap N)$ is solvable. (Note that $O^2(HN/N)=HN/N.$)

If the former holds, then $ H \leq HN \trianglelefteq G $, contradicting Claim~4 that the normal closure of $ H $ is all of $ G $. Hence, the latter must hold, i.e., $ H / (H \cap N) $ is solvable.
If $ N $ is solvable, then $ H \cap N $ is also solvable, which would imply $ H $ is solvable, a contradiction. Therefore, $ N $ is nonsolvable. Since $ H_G = 1 $ by Claim 5, we have $ N \not\leq H $ and consequently
$H \cap N < N.$  By Claim~1, either
$O^2(H \cap N) \trianglelefteq N$ \text{and}  $N / O^2(H \cap N) \text{ is a } 2\text{-group},$
or $ H \cap N $ is solvable.
The latter is impossible, so the former must hold. However, this contradicts Remark \ref{lem:directproduct}.
Thus, our assumption that $ HN < G $ is false, and we conclude that $ G = HN $.

\smallskip
\textbf{Claim $7$.} $G$ is a finite nonabelian simple group.
\smallskip

Suppose by contradiction that $G$ is not a finite nonabelian simple group. Let $N$ be a minimal normal subgroup of $G$. By Claim 6, $G = HN$. Since $N$ is a minimal normal subgroup of $G$, either $N$ is elementary abelian or 
$N \cong S^k,$
where $S$ is a nonabelian simple group and $k \geq 1$. Note that since $H$ is nonsolvable, it must have even order by Feit-Thompson's odd order theorem (see \cite{FT}). 

Assume first that $N$ is abelian. We claim that $H \cap N = 1$. Note that $H \cap N \trianglelefteq H$ and since $N$ is abelian, $N$ normalizes $H \cap N$. In particular, as $G = HN$, we have $H \cap N \trianglelefteq G$. Therefore, by Claim 5,  $H\cap N\leq H_G=1$.

We will show that $N$ does not contain elements of odd order. Suppose $1 \neq x \in N$ is an odd order element. Let $1 \neq h \in H$ be a $2$-element. Then $xh \in xH$ and hence $xh$ has odd order. However, in the quotient group $G/N$, we have $N(xh) = Nh,$
which is a nontrivial $2$-element in $G/N$ since $h \notin N$ is a $2$-element. This contradiction shows that $N$ is a $2$-group.

Let $1 \neq a \in H$ be an element of odd order. We will show that $a^n\in H$ for all $n\in N$. For a contradiction, suppose   that $a^n \notin H$ for some $n \in N$. Then 
$a^n a^{-1} \in a^n H,$
so $a^n a^{-1}$ has odd order. However,
$a^n a^{-1} = n^{-1} (a n a^{-1}) \in N,$
since $n \in N \trianglelefteq G$. So $a^n a^{-1}$ is also a $2$-element. Hence
$a^n a^{-1} = 1,$
and therefore $a^n = a \in H$, a contradiction.
Thus $a^n \in H$ for all $n \in N$ and all $a \in H$ of odd order. Let $\mathcal{U}$ be the set of all odd order elements in $H$. It follows that $\mathcal{U}^n = \mathcal{U}$ for all $n\in N$, and so $\left<\mathcal{U}\right> = \left<\mathcal{U}^n\right> = \left<\mathcal{U}\right>^n$, hence $N$ normalizes the subgroup $\left<\mathcal{U}\right>$ of $H$. Moreover, as $\mathcal{U}$ is closed under the conjugation action of $H$, $H$ also normalizes the subgroup $\left<\mathcal{U}\right>$. Thus, $\left<\mathcal{U}\right> \unlhd HN = G$, and by Claim 5, as $\left<\mathcal{U}\right>\leq H$, we have $\left<\mathcal{U}\right> \leq H_G=1$. Hence, $H$ does not contain any nontrivial elements of odd order, so $H$ is a $2$-group which is solvable, a contradiction.

Therefore,
$N \cong S^k$,
where $S$ is a nonabelian simple group and $k \geq 1$. Let $U = H \cap N$. Then $U < N$. If every element in $N \setminus U$ has even order, then by Lemma \ref{lem:p-singular}, 
$O^2(U) \trianglelefteq N$
and $N / O^2(U)$ is a nontrivial $2$-group, which is impossible by Remark \ref{lem:directproduct}. Thus we can find $x \in N \setminus U$ of odd order. Note that $U \trianglelefteq H$.

Assume first that $U = 1$. Let $z \in H$ be an involution. Then $z \notin N$. As the pair $(G, H)$ satisfies condition $\OO$, the coset $xH$ contains only odd order elements. Hence, $xz$ has odd order. However, in the quotient group $G/N$, $N(xz) = Nz$ is an involution as $z$ is an involution in $G\setminus N$, a contradiction.

Assume that $U > 1$. Then $(N, U)$ satisfies condition $\OO$. Also, by Claim 1, either
$O^2(U) \trianglelefteq N $  {and} $ N / O^2(U) \text{ is a } 2\text{-group},$
or $U$ is solvable. By Remark \ref{lem:directproduct}, the former case cannot happen. Hence $U$ is solvable.
Observe that $U < H$ since $\langle H^G \rangle = G$ so $H \not\leq N$. If $H / U$ is of odd order, then it is solvable and thus $H$ is solvable, a contradiction. Hence $H / U$ is of even order. Thus we can find a $2$-element $y \in H \setminus U$. We have
$
x y \in x H,$
so $x y$ has odd order. However, in the quotient group $G/N$, $N(xy) = Ny$ is a nontrivial $2$-element, which is a contradiction.

\smallskip
\textbf{Claim $8$.}  For every $\sigma\in\Aut(G)$, the pair $(G,\sigma(H))$ satisfies condition $\OO$. Consequently, if $x\in \Delta_H(G)$ has odd order, then $x^G y^G$ contains only odd order elements for every $y\in H.$ Furthermore, if $x\in \Delta_H(G)$ has even order, then $x^G y^G$ contains only even order elements for every $y\in H.$ Recall that $\Delta_H(G)=G\setminus \cup_{g\in G}H^g$.

\smallskip
We first show that the pair $(G,\sigma(H))$ satisfies condition $\OO$ for every $\sigma\in\Aut(G)$. Fix $\sigma\in\Aut(G)$. Note that $\sigma(H)$ is a proper subgroup of $G$. Let $x\in G\setminus \sigma(H)$ be an element of odd order. Then $t=\sigma^{-1}(x)\in G\setminus H$ has odd order. Let $a\in\sigma(H)$. Then $a=\sigma(h)$ for some $h\in H.$ As $(G,H)$ satisfies condition $\OO$, we see that $th$ has odd order. It follows that $\sigma(th)=\sigma(t)\sigma(h)=xa$ also has odd order and the claim follows.

Let $x\in \Delta_H(G)$ be of odd order. Then  $x^G\cap H=\emptyset$.  Let $y\in H$. Suppose by contradiction that there exists an element $z\in x^Gy^G$ which has even order. Then $z=x^uy^v$ for some $u,v\in G.$ By the previous claim, the pair $(G,H^v)$ satisfies condition $\OO$. Since $x\in \Delta_H(G)$, $x^u\in G\setminus H^v$ has odd order. As $y^v\in H^v$, condition $\OO$ implies that $z=x^uy^v\in x^uH^v$ has odd order, which is a contradiction. Therefore, $x^Gy^G$ consists entirely of odd order elements, for every odd order element $x\in \Delta_H(G)$ and $y\in H.$
For the remaining claim, the proof is similar.

\smallskip
\textbf{Claim $9$.} Let $x\in \Delta_H(G)$ and $h\in H$. Suppose that either $o(x)$ is odd and $o(h)$ is even or $o(x)$ is even and $o(h)$ is odd. Then $x^Gh^G\subseteq \Delta_H(G)$. 

\smallskip
First, assume that $o(x)$ is odd and $o(h)$ is even. By way of contradiction, suppose that $x^ah^b\not\in\Delta_H(G)$ for some $a,b\in G.$ Then $(x^ah^b)^g\in H$ for some $g\in G.$ Thus $z=x^uh^v\in H$ for some $u,v\in G.$ It follows that $(x^{-1})^uz=h^v$. Since $(x^{-1})^u\in \Delta_H(G)$ has odd order and $z\in H$, by Claim  $8,$   $h^v=(x^{-1})^uz$ must have odd order, which contradicts the choice of $h.$ Therefore, $x^Gh^G\subseteq \Delta_H(G)$ as required. The case when $o(x)$ is even and $o(h)$ is odd can be argued similarly.

\smallskip
\textbf{Claim $10$.}  The final contradiction.
\smallskip

Since $H$ is a proper subgroup of a finite nonabelian simple group $G$, we deduce that $|G:H|>2$. Recall that $\Delta_H(G)$ is nonempty. 

(a) Assume first that  $\Delta_H(G)$ contains a derangement $x$ of odd order.   Since $H$ is nonsolvable by hypothesis, we can find an element $1\neq h\in H$ of even order. Let $C=h^G$ and $D=x^G$. By Claim 9, $DC\subseteq \Delta_H(G)$ and by Claim 8 the product $DC$ contains only elements of odd order. Let $y\in DC.$ Then $y\in \Delta_H(G)$ is an odd order derangement and so, arguing as above, $y^GC\subseteq \Delta_H(G)$ which contains only odd order derangements. Thus $DC^2$ contains only odd order derangements. Repeating this argument, for each integer $k\ge 1$,  $DC^k$ contains only odd order derangements of $G$. By the Basic Covering Theorem (see \cite[Theorem 1.1]{AH}), since $C\neq 1$, there exists a positive integer $m$ such that $C^m=G$. It follows that $DC^m=DG=G\subseteq \Delta_H(G)$, which is a contradiction as the identity element in $G$ is not a derangement.

(b) Assume that every derangement in $\Delta_H(G)$ has even order and fix $x\in \Delta_H(G)$. As $H$ is nonsolvable, $H$ is not a $2$-group and thus $|H|$ is divisible by an odd prime $p$. By Cauchy's theorem, let $h\in H$ be an element of order $p$. As in the previous case, let $D=x^G$ and $C=h^G.$ By Claims 8 and 9, $DC\subseteq \Delta_H(G)$ and it contains only even order derangements. Hence for each integer $k\ge 1$, the product $DC^k$ contains only even order derangements. By \cite[Theorem 1.1]{AH}, we have $C^m=G$ for some positive integer $m$ and thus $DC^m=G\subseteq \Delta_H(G)$, which leads to a contradiction as before. This contradiction completes the proof of the claim.
\end{proof}

\begin{proof}[\textbf{Proof of Theorem \ref{th:odd-order-coset}}]
Let $G$ be a finite group and let $H$ be a proper subgroup of $G$. Suppose that for every element $x\in G\setminus H$ of odd order, all elements in the coset $xH$ have odd order. By Theorem \ref{th:odd-order},  either $O^2(H)$ is normal in $G$ or $H$ is solvable.
If $H$ is solvable, then the theorem follows. So, assume that $H$ is nonsolvable. Then $O^2(H)\unlhd G.$  If $G/O^2(H)$ is a $2$-group, then we are done. So, assume that $G/O^2(H)$ is not a $2$-group. Then we can find an odd order element $x\in G\setminus O^2(H)$. By hypothesis, $xH$ contains entirely of odd order elements and so $xO^2(H)$ contains only elements of odd order. By \cite[Theorem B]{GN}, $O^2(H)$ is solvable and so is $H$, which is a contradiction.
\end{proof}

\begin{proof}[\textbf{Proof of Corollary \ref{cor1}}]
First, observe that the pair $(G,H)$ satisfies the hypothesis of Theorem \ref{th:odd-order-coset}. Hence either $H$ is solvable or $O^2(H)\unlhd G$ and $G/O^2(H)$ is a $2$-group. If $H$ is solvable, then the conclusion holds. So, assume $H$ is not solvable. As $G/O^2(H)$ is a $2$-group, we see that $O^2(G)\leq O^2(H)\leq H$ and clearly $H$ is subnormal in $G$ since $G/O^2(G)$ is a $2$-group. This establishes (a).  Since $H$ is proper and subnormal in $G$, we must have $H<N_G(H)$ and so $(N_G(H),H)$ is an equal order pair. Finally, let $x\in G\setminus H$. We claim that $x$ is a $2$-element. Write $x=x_2x_{2'}$, where $x_2,x_{2'}$ are the $2$-part and $2'$-part of $x$, respectively. As $G/O^2(G)$ is a $2$-group, we have $x_{2'}\in O^2(G)\leq H$. Moreover, as $x\not\in H$, $x_2\not\in H$.  It follows that $xH=x_2H$ and hence $x$ and $x_2$ have the same order. In particular, $x$ is a $2$-element as wanted.
\end{proof}

\begin{proof}[\textbf{Proof of Theorem \ref{cor2}}]
Let $G$ be a counterexample to the theorem with minimal order. Then there exists a $p$-element $x\in G$ such that $xy$ is  $p$-regular for all nontrivial $p$-regular elements $y\in G$, but $x\not\in O_p(G).$ 
Let $L=\la x^G\ra$. Suppose $L<G$. Then by the minimality of $|G|$, $x\in O_p(L)$.  As $L\unlhd G$,  we have $O_p(L)\subseteq O_p(G)$ and thus $x\in O_p(G)$,   a contradiction. Hence, we may assume $G=\la x^G\ra.$

Let $1<N\unlhd G$. Then $xN$ is a $p$-element in $G/N$, and for all nontrivial $p$-regular elements $yN\in G/N$, we can find a nontrivial $p$-regular element $g\in G$ such that $yN=gN$ and thus $xyN=xgN$ is  $p$-regular in $G/N$. Therefore,  the pair $(G/N,xN)$ satisfies the hypothesis of the theorem, and by the minimality of $|G|$, $xN\in O_p(G/N)$. 
Suppose $O_p(G)>1$. Then since $O_p(G/O_p(G))$ is trivial and by the previous claim $xO_p(G)$ lies in $O_p(G/O_p(G))$, which forces $x\in O_p(G)$, a contradiction. So, we may assume $O_p(G)=1$.

Let $N$ be a minimal normal subgroup of $G$. Then $N$ is not a $p$-group. Choose a nontrivial $p$-regular element $y\in N$. Then $xyN=xN$. Since $xN$ is a $p$-element in $G/N$ and  $xy$ is  $p$-regular by hypothesis,  it follows that $x\in N$. But $G=\la x^G\ra$, so $G=N$. Hence, $G$ is a finite nonabelian simple group with $p$ dividing $|G|$. 

Let $C=x^G$ and $D=y^G$ for some nontrivial $p$-regular element $y\in G.$ By assumption, every element in $CD$ is $p$-regular. Since $x$ and $y$ have coprime orders, every element in $CD$ is nontrivial. By hypothesis again, $C^2D=C(CD)$ contains only nontrivial $p$-regular elements. Inductively, for each integer $k\ge 1$, the product $C^kD$ contains only nontrivial $p$-regular elements. Since $|C|>1$ as $x$ is nontrivial, by \cite[Theorem 1.1]{AH}, $C^m=G$ for some positive integer $m.$ It follows that $C^mD=GD=G$ consists of only nontrivial $p$-regular elements, which is a contradiction. The proof is now complete.
\end{proof}

\section*{Acknowledgment}  The authors are grateful to the reviewer for the helpful comments and  suggestions.


\end{document}